\documentclass{amsart}

\usepackage[latin1]{inputenc}
\usepackage[T1]{fontenc}
\usepackage{enumerate}
\usepackage{xcolor,pgf,tikz,pgflibraryarrows,pgffor,pgflibrarysnakes}
\usepackage{float}

\newtheorem{theorem}{Theorem}[section]

\newtheorem{claim}{Claim}
\newtheorem{pbm}{Problem}
\newtheorem{question}{Question}

\theoremstyle{definition}

\theoremstyle{remark}

\newcommand{\N}{\mathbb N}
\newcommand{\K}{\mathbb K}

\newcommand{\C}{\mathbb C}
\newcommand{\U}{\mathcal U}

\numberwithin{equation}{section}

\begin{document}

\title[Linear chaos and frequent hypercyclicity]{Linear chaos and frequent hypercyclicity}

\author[Q. Menet]{Quentin Menet}
\address{Département de Mathématique\\
Université de Mons\\
20 Place du Parc\\
7000 Mons, Belgique}
\email{Quentin.Menet@umons.ac.be}
\thanks{The author is a postdoctoral researcher of the Belgian FNRS}


\maketitle

\begin{abstract}
We answer one of the main current questions in Linear Dynamics by constructing a chaotic operator on $\ell^1$ which is not $\U$-frequently hypercyclic and thus not frequently hypercyclic. This operator also gives us an example of a chaotic operator which is not distributionally chaotic. We complement this result by showing that every chaotic operator is reiteratively hypercyclic.
\end{abstract}
\maketitle

\section{Introduction}


Let $X$ be a separable infinite-dimensional Fréchet space and $T$ a continuous linear operator on $X$.
We say that $T$ is  \emph{hypercyclic} if there exists a vector $x\in X$ (also called hypercyclic) such that its orbit $\text{Orb}(x,T):=\{T^nx:n\ge 0\}$ is dense in $X$. In other words, $T$ is hypercyclic if there exists a vector $x\in X$ such that for any non-empty open set $U\subset X$ the return set $N(x,U):= \{n\ge 0:T^nx\in U\}$ is non-empty. 

The first example of a hypercyclic operator was given by Birkhoff~\cite{Birkhoff} in 1929. He showed that the translation operators $T_a$ on the space of entire functions $H(\C)$ defined by $T_af(z)=f(z+a)$ are hypercyclic if and only if $a\ne 0$. Another important family of hypercyclic operators was given by Salas~\cite{Salas} who showed that a weighted shift $B_w$ on $\ell^p$ is hypercyclic if and only if the sequence $(w_1\cdots w_n)_n$ is unbounded.

Hypercyclic operators have been actively investigated over the last three decades (see \cite{2Bayart2,2Grosse2}). For instance, we now know that the set of hypercyclic vectors for $T^n$ coincides with the set of hypercyclic vectors for $T$~\cite{Ansari}, that each separable infinite-dimensional Fréchet space supports a hypercyclic operator~\cite{4Bonet2},  that there exists a hypercyclic operator $T$ such that $T\oplus T$ is not hypercyclic~\cite{2Read},... 

In the last decade, attention has been given to the frequency with which the orbit of a hypercyclic vector meets each non-empty open set. In 2004, Bayart and Grivaux~\cite{2Bayart0} introduced the notion of frequently hypercyclic operators. An operator $T$ is said to be \emph{frequently hypercyclic} if there exists a vector $x\in X$ (also called frequently hypercyclic) such that for any non-empty open set $U\subset X$ the return set $N(x,U)$ is a set of positive lower density, where the lower density of a set $A$ of non-negative integers  is given by
\[\underline{\text{dens}}(A):=\liminf_{N\to \infty}\frac{\#(A\cap\mathopen[0,N\mathclose])}{N+1}.\]
Several classical hypercyclic operators are in fact frequently hypercyclic~\cite{2Bayart}: the translation operators on $H(\C)$, the derivative operator on $H(\C)$,... We even know a characterization of frequently hypercyclic weighted shifts on $\ell^p$ thanks to Bayart and Ruzsa~\cite{BayartR}: a weighted shift $B_w$ on $\ell^p$ is frequently hypercyclic if and only if
\[\sum_{n=1}^{\infty}\frac{1}{|w_1\cdots w_n|^p}<\infty.\]

Two other quantifications of the frequency of visits of an orbit were investigated by Shkarin~\cite{Shkarin} and by Bès, Peris, Puig and the author~\cite{Bes}: $\U$-frequent hypercyclicity and reiterative hypercyclicity. An operator $T$ is said to be \emph{$\U$-frequently hypercyclic} (resp. \emph{reiteratively hypercyclic}) if there exists a vector $x\in X$ such that for any non-empty open set $U\subset X$ the return set $N(x,U)$ is a set of positive upper density (resp. a set of positive upper Banach density). We recall that the upper density of a set $A$ is given by
\[\overline{\text{dens}}(A):=\limsup_{N\to \infty}\frac{\#(A\cap\mathopen[0,N\mathclose])}{N+1}\]
and the upper Banach density of $A$ is given by
\[\overline{\text{Bd}}(A):=\lim_{N\to \infty}\frac{a_N}{N}\quad \text{with}\quad a_N:=\limsup_k\#(A\cap \mathopen[k+1,k+N\mathclose]).\]
We remark that by definition every frequently hypercyclic operator is $\U$-frequently hypercyclic and that every $\U$-frequently hypercyclic operator is reiteratively hypercyclic.

On the other hand, one can be interested in the existence of periodic vectors, i.e. the existence of vectors $x$ for which there exists $N>0$ such that $T^Nx=x$. The behaviour of the orbit of a periodic point is obviously very different from the behaviour of the orbit of a hypercyclic vector. In fact, a hypercyclic operator with a dense set of periodic points is said to be \emph{chaotic}. It means that $T$ is chaotic if and only if we can find in each non-empty open set some vector with a dense orbit and some periodic vector. The translation operators and the derivative operator on $H(\C)$ are examples of chaotic operators~\cite{3God}. A characterization of chaotic weighted shifts is also known~\cite{2Grosse3,2Martinez}. In particular, we know thanks to Bayart and Ruzsa~\cite{BayartR} that a weighted shift on $\ell^p$ is chaotic if and only if it is frequently hypercyclic.

One can wonder if there exists a link between frequent hypercyclicity and chaos. Indeed, on the one hand, we have the existence of vectors which visit frequently each non-empty open set and, on the other hand, we have the existence of vectors which visit each non-empty open set and of vectors which visit infinitely the same vectors. Moreover each of these two notions is related to the existence of sufficiently many eigenvectors associated to eigenvalues of modulus~$1$. For instance, if $X$ is a complex vector space, it is well known that the set of periodic points of $T$ is given by
\[\text{span}\{x\in X: Tx=\lambda x\ \text{for some root of unity $\lambda\in \mathbb{C}$}\}.\]
On the other hand, Grivaux~\cite{Grivaux11} showed that if  $X$ is a complex Banach space and if the eigenvectors of $T$ associated to eigenvalues of modulus~$1$ are perfectly spanning, then $T$ is frequently hypercyclic. A first answer was given by Bayart and Grivaux~\cite{BGrivaux}. They showed that there exists a weighted shift on $c_0$ which is frequently hypercyclic but not chaotic. However we cannot hope to construct a chaotic weighted shift on $c_0$ which is not frequently hypercyclic since each chaotic weighted shift on $c_0$ is frequently hypercyclic~\cite{4Bonilla}. It is in fact an important open problem in Linear Dynamics to know if every chaotic operator is frequently hypercyclic (see \cite[Chapter 6]{2Bayart2} and \cite[Chapter 9]{2Grosse2}).
\begin{pbm}
Is every chaotic operator frequently hypercyclic?
\end{pbm}

This question has been first posed by Bayart and Grivaux in \cite[Question 6.4]{BGrivaux} and can be found in many papers~\cite{BM14, JBes, 4Bonilla, RFGP, Grivaux11, KGE7}. Obviously, one can also wonder if there is a link between chaos and $\U$-frequent hypercyclicity or between chaos and reiterative hypercyclicity.

\begin{pbm}
Is every chaotic operator $\U$-frequently/reiteratively hypercyclic?
\end{pbm}

The notion of chaos that we mentioned previously  is generally called chaos in the sense of Devaney in order to make a clear distinction with the other notions of chaos found in the mathematical literature. The first notion of chaos was given by Li and Yorke~\cite{Li}. An operator $T$ is said to be \emph{Li-Yorke chaotic} if there exists an uncountable subset $\Gamma\subset X$ such that for every $x, y\in \Gamma$, $x\ne y$, we have
\[\liminf_n \|T^nx-T^ny\|=0\quad \text{and}\quad \limsup_n \|T^nx-T^ny\|>0,\]
where $\|\cdot\|$ is a F-norm inducing the topology of $X$. 

In 1994, Schweizer and Sm\'{i}tal~\cite{Schweizer} extended the notion of Li-Yorke chaotic operators by introducing the notion of distributionally chaotic operators. An operator $T$ is said to be \emph{distributionally chaotic} if there exist an uncountable subset $\Gamma\subset X$ and $\varepsilon>0$ such that for every $x\ne y\in \Gamma$, for every $\tau>0$, we have
\[\underline{\text{dens}}\{n\ge 0:\|T^nx-T^ny\|<\varepsilon\}=0\quad \text{and}\quad \overline{\text{dens}}\{n\ge 0:\|T^nx-T^ny\|<\tau\}=1.\]

In \cite{Bernardes}, the authors study the notion of distributionally chaotic operators and pose the following question.

\begin{pbm}{\cite[Problem 37]{Bernardes}}\label{pbm 2}
Are there chaotic operators which are not distributionally chaotic?
\end{pbm}

The goal of this paper consists in bringing a complete answer to each of the problems mentioned above thanks to the following two theorems.

\begin{theorem}\label{chaosreit}
Every chaotic operator on a separable infinite-dimensional Fréchet space is reiteratively hypercyclic.
\end{theorem}

\begin{theorem}\label{chaosfhc}
There exists a chaotic operator $T$ on $\ell^1$ which is neither $\U$-frequently hypercyclic nor distributionally chaotic. In particular, $T$ is chaotic and not frequently hypercyclic.
\end{theorem}


\section{Proof of Theorem~\ref{chaosreit}}
Let $X$ be a separable infinite-dimensional Fréchet space, $T$ a chaotic operator on $X$ and $x$ a hypercyclic vector for $T$. We show that for any non-empty open subset $U$ of $X$, the return set $N(x,U)$ is a set of positive upper Banach density.

Let $U$ be a non-empty open subset of $X$. Since $T$ is chaotic, there exist a periodic point $z\in U$ and a positive integer $d$ such that $T^dz=z$. We remark that for every $n\ge 0$ the set $U_n:=\bigcap_{l=0}^nT^{-ld}U$ is a non-empty open set since $U_n$ contains $z$ and $T$ is continuous. In particular, for every $n\ge 0$, the set $N(x,U_n)$ is non-empty, i.e. there exists $k_n\ge 0$ such that $T^{k_n+ld}x\in U$ for every $l\le n$. This implies that
\[N(x,U)\supset \bigcup_{n\ge 0}\{k_n+ld: 0\le l\le n\}.\]
We conclude that $\overline{\text{Bd}}(N(x,U))\ge \frac{1}{d}$ and thus that $T$ is reiteratively hypercyclic.\qed

\section{Proof of Theorem~\ref{chaosfhc}}\label{nonexist}
\subsection{Construction of the operator $T$}
Several important problems in Linear Dynamics have been solved thanks to the construction of a convenient upper-triangular perturbation of a weighted forward shift. For instance, Read~\cite{Read} constructed an upper-triangular perturbation of a weighted forward shift on $\ell^1$ for which every non-zero vector is cyclic and thereby solved in the negative the invariant subspace problem for $\ell^1$. In 2006, De La Rosa and Read~\cite{2Read} solved in the negative the Hypercyclicity Criterion problem by constructing a Banach space $X$ and a hypercyclic operator $T$ on $X$ such that $T$ is not weakly mixing, i.e. $T\oplus T$ is not hypercyclic. One could wonder if such operators also exist on some classical Banach spaces. The answer was given by Bayart and Matheron~\cite{BayartM} who constructed a hypercyclic operator $T$ on $\ell^1$ such that $T\oplus T$ is not hypercyclic by considering a convenient upper-triangular perturbation of a weighted forward shift.

More precisely, Bayart and Matheron consider an increasing sequence of non-negative integers $(b_n)_{n\ge 0}$ with $b_0=0$  and an operator $T$ of the form
\[Te_k=
\left\{\begin{array}{cl}
 w_{k+1}e_{k+1} & \quad\text{if}\ k\in \mathopen[b_n,b_{n+1}-1\mathclose[\\
\varepsilon_n e_{b_n}+f_n & \quad \text{if}\ k=b_{n}-1
   \end{array}\right.\]
   where $f_n=\sum_{k=0}^{b_{n}-1}f_{n,k}e_k$ and $\varepsilon_n>0$.
They then prove that for a convenient choice of sequences $(b_n)$, $(w_n)$, $(\varepsilon_n)$ and $(f_n)$, $T\oplus T$ is not hypercyclic and $e_0$ is a hypercyclic vector for $T$. We can remark that the positivity of each real number $\varepsilon_n$ is an essential assumption if we want that $e_0$ can have a dense orbit.

The starting point of our construction will be a little bit different since we will construct an operator $T$ such that each vector $e_n$ is a periodic point of~$T$. We will thus consider an operator $T$ of the form 
\[Te_k=
\left\{\begin{array}{cl}
 w_{k+1}e_{k+1} & \quad\text{if}\ k\in \mathopen[b_n,b_{n+1}-1\mathclose[\\
 f_n & \quad \text{if}\ k=b_{n}-1
   \end{array}\right.\]
   where $f_n=\sum_{k=0}^{b_{n}-1}f_{n,k}e_k$.

More precisely, we consider the operator defined as follows:

\[Te_k=
\left\{\begin{array}{cl}
 2e_{k+1} & \quad\text{if}\ k\in \mathopen[b_n,b_n+\delta_n\mathclose[\\
 e_{k+1}& \quad\text{if}\ k\in \mathopen[b_n+\delta_n,b_{n+1}-1\mathclose[\\
 \frac{1}{2^{\tau_n}}e_{b_{\varphi(n)}}-\frac{1}{2^{\delta_n}}e_{b_n} & \quad \text{if}\ k=b_{n+1}-1\ \text{with} \ n\ge 1\\
 -e_0& \quad \text{if}\ k=b_1-1
   \end{array}\right.\]
   where
   \begin{itemize}
   \item $\varphi$ is a map from $\N_0$ to $\N_0$ satisfying $\varphi(0)=0$ and for any $n\ge 0$
      \begin{eqnarray}
       \varphi(n+1)<n+1 \quad \text{and} \quad \#\varphi^{-1}(n)=\infty;
       \label{varphi}
   	  \end{eqnarray}
   \item $(\delta_n)_{n\ge 0}$ and $(\tau_n)_{n\ge 1}$ are increasing sequences of positive integers with ${\delta_0=0}$ satisfying for any $n\ge 1$
      \begin{eqnarray}
       \delta_n-\tau_n\to \infty;
       \label{deltatau1}
      \end{eqnarray}   
      \begin{eqnarray}
       \tau_n\ge\delta_{n-1}+2(n+1);
       \label{deltatau2}
      \end{eqnarray}  
   \item $(b_n)_{n\ge 0}$ is an increasing sequence of integers with $b_0=0$ such that for any $n\ge 1$
      \begin{eqnarray}
       b_{n+1}-b_n=2N_n(b_n-b_{n-1}) \ \text{for some integer $N_n\ge 1$};
       \label{bmultiple}
   	  \end{eqnarray} 
   	   \begin{eqnarray}
       2\delta_n <b_{n+1}-b_n;
       \label{b1} 
   	  \end{eqnarray}
   	  \begin{eqnarray}
   	  \frac{\delta_n}{b_{n+1}-b_n}\searrow 0.
   	  \label{b2}
   	  \end{eqnarray} 	  
   \end{itemize}

Each of these assumptions is satisfied if we consider, for instance, for every $n\ge 1$
\[\tau_n= 4^{n+1},\quad \delta_n=2 \tau_n\quad\text{and}\quad b_{n}-b_{n-1}=4^{2n+1}.\]
We remark that the assumption \eqref{b1} implies that $b_n+\delta_n\le b_{n+1}-1$. The vector $Te_k$ is thus well-defined for any $k\ge 0$. Moreover, since for any $k\ge 0$ we have $\|T e_k\|\le 2$, we can extend the definition of $T$ to $\ell^1$ by letting for any $(x_k)_{k\ge 0}\in \ell^1$
\[T\Big(\sum_{k=0}^{\infty}x_k e_k\Big)=\sum_{k=0}^{\infty}x_kT(e_k)\]
and we conclude that $T$ is a continuous operator on $\ell^1$ with $\|T\|= 2$.

The remainder of this section is devoted to showing that $T$ is chaotic (Section~\ref{Tchaotic}) and that $T$ is neither $\U$-frequently hypercyclic (Section~\ref{TpasUfhc}) nor distributionally chaotic (Section~\ref{Tpasdist}).

\subsection{$T$ is chaotic}\label{Tchaotic}
We first prove that each vector $e_k$ is a periodic point of $T$.

\begin{claim}\label{claimper} Let $n \ge 0$. If $k\in \mathopen[b_n,b_{n+1}\mathclose[$ then
\begin{eqnarray}
T^{2(b_{n+1}-b_n)}e_k=e_k.
\label{period}
\end{eqnarray}
\end{claim}
\begin{proof}
If $k\in \mathopen[b_0,b_1\mathclose[$, then by definition of $T$, we have
\[T^{b_1-b_0}(e_k)=-e_k\quad \text{and thus}\quad T^{2(b_1-b_0)}{e_k}=e_k.\]
It remains to prove that for every $n\ge 1$, if \eqref{period} is satisfied for every $l\in \mathopen[0,b_{n}\mathclose[$ then \eqref{period} is satisfied for every $k\in \mathopen[b_n,b_{n+1}\mathclose[$. 

Let $n\ge 1$ and $k\in \mathopen[b_n,b_{n+1}\mathclose[$. We assume that \eqref{period} is satisfied for every $l\in \mathopen[0,b_{n}\mathclose[$. If we let $j_k:=\#\mathopen[k,b_{n}+\delta_{n}\mathclose[=\#\mathopen[k-b_n,\delta_{n}\mathclose[$, we have by definition of $T$
\[T^{b_{n+1}-k}e_k=\frac{2^{j_k}}{2^{\tau_n}}e_{b_{\varphi(n)}}-\frac{2^{j_k}}{2^{\delta_n}}e_{b_n}.\]
Since $\varphi(n)<n$ and $b_{n+1}-b_{n}$ is a multiple of $2(b_{\varphi(n)+1}-b_{\varphi(n)})$, we deduce from our induction hypothesis that
\begin{align*}
T^{b_{n+1}-b_n}(T^{b_{n+1}-k}e_k)&=\frac{2^{j_k}}{2^{\tau_n}}e_{b_{\varphi(n)}}-\frac{2^{j_k}}{2^{\delta_n}}T^{b_{n+1}-b_n}e_{b_n}\\
&=\frac{2^{j_k}}{2^{\tau_n}}e_{b_{\varphi(n)}}-\frac{2^{j_k}}{2^{\delta_n}}\left(\frac{2^{\delta_n}}{2^{\tau_n}}e_{b_{\varphi(n)}}-e_{b_n}\right)\\
&=\frac{2^{j_k}}{2^{\delta_n}}e_{b_n}.
\end{align*}
We conclude that
\begin{align*}
T^{2(b_{n+1}-b_n)}e_k&=\frac{2^{j_k}}{2^{\delta_n}}T^{k-b_{n}}e_{b_n}\\
&=\frac{2^{j_k}}{2^{\delta_n}}2^{\#\{\mathopen[0,k-b_{n}\mathclose[\cap\mathopen[0,\delta_n\mathclose[\}}e_{k}\\
&=\frac{2^{j_k}}{2^{\delta_n}}2^{\delta_n-j_k}e_{k}=e_k.
\end{align*}
\end{proof}

We easily deduce the following assertion from \eqref{bmultiple} and Claim~\ref{claimper}.

\begin{claim}\label{claimper2} Let $x\in \ell^1$. If for any $k\ge b_{n+1}$ we have $x_k=0$ then \[T^{2(b_{n+1}-b_n)}x=x.\] In particular, every finite sequence is a periodic point of $T$.
\end{claim}

We now show that $T$ is hypercyclic. To this end, we first prove the following claim.

\begin{claim}\label{hyp0}
Let $\varepsilon>0$, let $k\ge 0$, $N\ge 1$ and $0\le M<N$ be integers and let $x_k\in \K$.
Then there exist $m\ge 0$, $l\ge 0$ and $z_m\in \K$ such that
\begin{eqnarray}
\label{epsilon}
|z_m|<\varepsilon \quad \text{and} \quad \|T^{lN+M}(z_me_m)-x_ke_k\|<\varepsilon.
\end{eqnarray}
\end{claim}
\begin{proof}
Let $n\ge 0$ such that $k\in \mathopen[b_n,b_{n+1}\mathclose[$.
In view of \eqref{varphi} and \eqref{deltatau1}, we can consider a positive integer $s$ such that $|x_k|<\varepsilon 2^s$ and a positive integer $t$ such that
\begin{eqnarray}\label{m}
\varphi(t)=n\quad\text{and}\quad \delta_t-\tau_t\ge s+N+b_{n+1}-b_n.
\end{eqnarray}
We then let $m=b_t+\delta_t-\tau_t-s-r$ where $0\le r<N$ satisfies \[(b_{t+1}-m+k-b_n)\mod N=M.\]
We remark that $m\in \mathopen[b_{t},b_{t}+\delta_t\mathclose]$. Indeed, we deduce from \eqref{m} that
\[b_t+\delta_t\ge m\ge b_t+\delta_t-\tau_t-s-N\ge b_t+b_{n+1}-b_n\ge b_t.\]
Moreover, since $k\in \mathopen[b_n,b_{n+1}\mathclose[$ and $b_{t+1}\ge b_t+\delta_t$, we have
$b_{t+1}-m+k-b_n\ge 0$.
In particular, we deduce that $b_{t+1}-m+k-b_n=lN+M$ for some $l\ge 0$.

Let $j:=\#(\mathopen[0,k-b_n\mathclose[\cap\mathopen[0,\delta_n\mathclose[)$ and $z_m=\frac{x_k}{2^{s+r+j}}$. By definition of $s$, we have
$|z_m|=\frac{|x_k|}{2^{s+r+j}}\le \frac{|x_k|}{2^s}<\varepsilon$.
It remains to show that 
\[\|T^{b_{t+1}-m+k-b_n}(z_me_m)-x_ke_k\|<\varepsilon.\]
By definition of $T$, we have $T^{b_{t+1}-m}e_m=2^{b_t+\delta_t-m}(\frac{1}{2^{\tau_t}}e_{b_{\varphi(t)}}-\frac{1}{2^{\delta_t}}e_{b_t})$ and thus
\begin{align*}
T^{b_{t+1}-m}(z_m e_m)&= \frac{x_k}{2^j}e_{b_n}-\frac{2^{\tau_t}}{2^{\delta_t+j}} x_k e_{b_t}.
\end{align*}
Since $k-b_n\le b_{n+1}-b_n<\delta_t$, it follows that
\[T^{b_{t+1}-m+k-b_n}(z_m e_m)=x_k e_{k}-2^{k-b_n}\frac{2^{\tau_t}}{2^{\delta_t+j}}x_k e_{b_t+k-b_n}.\]
We conclude by \eqref{m} that
\begin{align*}
\|T^{b_{t+1}-m+k-b_n}(z_me_m)-x_ke_k\|&\le 2^{k-b_n}\frac{2^{\tau_t}}{2^{\delta_t+j}}|x_k|
\le 2^{-s}|x_k|<\varepsilon.
\end{align*}
\end{proof}

\begin{claim}\label{hyp}
$T$ is hypercyclic.
\end{claim}
\begin{proof}
We recall that $T$ is hypercyclic if and only if $T$ is topologically transitive, i.e. for any non-empty open sets $U$, $V$, there exists $n\ge 0$ such that $T^n(U)\cap V\ne \emptyset$.

Let $x$, $y\in \ell^1$ be finite sequences and $\varepsilon>0$.
It thus suffices to  show that there exist $z\in \ell^1$ and $n\ge 0$ such that $\|z\|<\varepsilon$ and $\|T^n(y+z)-x\|<\varepsilon$. Moreover, since $y$ is a finite sequence, we know by Claim~\ref{claimper2} that there exists $N_0\ge 1$ such that $T^{N_0}y=y$. We therefore deduce that it suffices to find a sequence $z\in \ell^1$ and a multiple $n$ of $N_0$ such that $\|z\|<\varepsilon$ and $\|T^nz-(x-y)\|<\varepsilon$.

Let $\tilde{z}:=x-y=\sum_{k=0}^d \tilde{z}_ke_k$. Using Claim~\ref{hyp0} for $\frac{\varepsilon}{d+1}$, $k=0$, $N=N_0$, $M=0$ and $x_k=\tilde{z}_0$, we obtain
an integer $m_0\ge 0$, $z_{m_0}\in \mathbb{K}$ and $l_0\ge 0$ such that
\[|z_{m_0}|<\frac{\varepsilon}{d+1}\quad \text{and}\quad \|T^{l_0N_0}(z_{m_0}e_{m_0})-\tilde{z}_{0}e_0\|<\frac{\varepsilon}{d+1}.\]
We then use Claim~\ref{hyp0} for $\frac{\varepsilon}{d+1}$, $k=1$, $N=N_1$, $M=l_0N_0$ and $x_k=\tilde{z}_1$, where $N_1$ is a multiple of $N_0$, of the period of $e_{m_0}$ such that $N_1>l_0N_0$. We thus obtain
an integer $m_1\ge 0$, $z_{m_1}\in \mathbb{K}$ and $l_1\ge 0$ such that
\[|z_{m_1}|<\frac{\varepsilon}{d+1}\quad \text{and}\quad \|T^{l_1N_1+l_0N_0}(z_{m_1}e_{m_1})-\tilde{z}_{1}e_1\|<\frac{\varepsilon}{d+1}.\]
Since $N_1$ is a multiple of the period of $e_{m_0}$, we also deduce that
\[\|T^{l_1N_1+l_0N_0}(z_{m_0}e_{m_0})-\tilde{z}_{0}e_0\|=\|T^{l_0N_0}(z_{m_0}e_{m_0})-\tilde{z}_{0}e_0\|<\frac{\varepsilon}{d+1}.\] 
By using repeatedly Claim~\ref{hyp0}, we can in fact obtain $z_{m_0},\ldots,z_{m_d}\in \mathbb{K}$, $l_0,\ldots,l_{d}\ge 0$ and $N_0,\ldots,N_d$ such that for any $0\le k\le d$,
\begin{itemize}
\item $\displaystyle{|z_{m_k}|<\frac{\varepsilon}{d+1}}$;
\item $\displaystyle{\|T^{l_kN_k+\cdots+l_0N_0}(z_{m_k}e_{m_k})-\tilde{z}_{k}e_k\|<\frac{\varepsilon}{d+1}}$;
\item $N_k$ is a multiple of $N_0$ and of the periods of $e_{m_0},\ldots,e_{m_{k-1}}$.
\end{itemize}
Let $z=\sum_{k=0}^{d}z_{m_k}e_{m_k}$ and $n=\sum_{k=0}^{d}l_kN_k$. We conclude that $\|z\|<\varepsilon$, $n$ is a multiple of $N_0$ and
\begin{align*}
\|T^nz-(x-y)\|&=\Big\|T^n\Big(\sum_{k=0}^{d}z_{m_k}e_{m_k}\Big)-\sum_{k=0}^d \tilde{z}_ke_k\Big\|\\
&\le \sum_{k=0}^{d}\|T^{\sum_{j=0}^{d}l_jN_j}(z_{m_k}e_{m_k})-\tilde{z}_ke_k\|\\
&= \sum_{k=0}^{d}\|T^{\sum_{j=0}^{k}l_jN_j}(z_{m_k}e_{m_k})-\tilde{z}_ke_k\|<\sum_{k=0}^{d} \frac{\varepsilon}{d+1}=\varepsilon.
\end{align*}
\end{proof}

It directly follows from Claim~\ref{claimper2} and Claim~\ref{hyp} that \textbf{$\mathbf{T}$ is chaotic}.

\subsection{$T$ is not $\U$-frequently hypercyclic}\label{TpasUfhc}

Let $x\in \ell^1$ and $n\ge 0$. We let  $P_n$ be the operator defined on $\ell^1$ by
\[P_nx=\sum_{k=b_n}^{b_{n+1}-1}x_ke_k\]
and we let \[X_{n}:=\sum_{k=b_n}^{b_{n+1}-1}2^{\#\mathopen[k-b_n,\delta_n\mathclose[}x_ke_k.\]
In order to show that $T$ is not $\U$-frequently hypercyclic, we
start by proving three claims concerning the norms of the elements $P_nT^jP_lx$.
We already remark that if $n>l$ then, by definition of $T$, we have for any $j\ge 0$,
$P_nT^jP_lx=0$.
In particular, we have for every $n\ge 0$, every $j\ge 0$,
\[P_nT^jx=\sum_{l\ge n}P_nT^jP_lx.\]

\begin{claim}\label{fhc0}
Let $x\in \ell^1$ and $0\le n<l$. Then
\[\sup_{j\ge 0}\|P_nT^jP_lx\|\le \frac{1}{2^{2(l+1)}}\|X_{l}\|.\]
\end{claim}
\begin{proof}
We first remark that if $n\ne \varphi^N(l)$ for every $N\ge 1$, then for any $j\ge 0$, we have by definition of $T$
\[P_nT^jP_lx=0.\]
Suppose that $n=\varphi^M(l)$ for some $M\ge 1$. 
We can then prove that for any $k\in \mathopen[b_l,b_{l+1}\mathclose[$, we have
\begin{eqnarray}\label{eqpnl}
\sup_{j\ge 0}\|P_nT^j e_k\|\le \frac{2^{\#\mathopen[k-b_l,\delta_l\mathclose[}}{2^{\tau_l}}\sup_{j\ge 0}\|P_nT^j e_{b_{\varphi(l)}}\|.
\end{eqnarray}
Indeed, if $j<b_{l+1}-k$ then $P_nT^j e_k=0$. If $j\in\mathopen[b_{l+1}-k,b_{l+1}-k+b_{l+1}-b_l\mathclose[$ then
\[\|P_nT^j e_k\|=\Big\|P_nT^{j-(b_{l+1}-k)}\Big(\frac{2^{\#\mathopen[k-b_l,\delta_l\mathclose[}}{2^{\tau_l}}e_{b_{\varphi(l)}}\Big)\Big\|\le \frac{2^{\#\mathopen[k-b_l,\delta_l\mathclose[}}{2^{\tau_l}}\sup_{i\ge 0}\|P_nT^i e_{b_{\varphi(l)}}\|\]
and if $j\in \mathopen[b_{l+1}-k+b_{l+1}-b_l, 2(b_{l+1}-b_l)\mathclose[$, we have $P_nT^j e_k=0$ since \[T^{b_{l+1}-k+b_{l+1}-b_l} e_k=\frac{2^{\#\mathopen[k-b_l,\delta_l\mathclose[}}{2^{\delta_l}}e_{b_l}.\]
We conclude that \eqref{eqpnl} is satisfied because $T^{2(b_{l+1}-b_l)}e_k=e_k$ (Claim~\ref{claimper}).

Let $N:=\min\{M\ge 1: n=\varphi^M(l)\}$. We then get thanks to \eqref{varphi}, \eqref{deltatau2} and \eqref{eqpnl} that
\begin{align*}
\sup_{j\ge 0}\|P_nT^jP_lx\|&=\sup_{j\ge 0}\Big\|P_nT^j\Big(\sum_{k=b_l}^{b_{l+1}-1}x_ke_k\Big)\Big\|\\
&\le \sum_{k=b_l}^{b_{l+1}-1}|x_k|\sup_{j\ge 0}\|P_nT^j e_k\|\\
&\le \sum_{k=b_l}^{b_{l+1}-1}|x_k|\frac{2^{\#\mathopen[k-b_l,\delta_l\mathclose[}}{2^{\tau_l}}\sup_{j\ge 0}\|P_nT^j e_{b_{\varphi(l)}}\|\\
&\vdots\\
&\le \sum_{k=b_l}^{b_{l+1}-1}|x_k|\frac{2^{\#\mathopen[k-b_l,\delta_l\mathclose[}}{2^{\tau_l}}\prod_{s=1}^{N-1}\frac{2^{\delta_{\varphi^s(l)}}}{2^{\tau_{\varphi^s(l)}}}\sup_{j\ge 0}\|P_nT^j e_{b_{n}}\|\\
&=\frac{\|X_l\|}{2^{\tau_l}}\prod_{s=1}^{N-1}\frac{2^{\delta_{\varphi^s(l)}}}{2^{\tau_{\varphi^s(l)}}}2^{\delta_n}
=\frac{2^{\delta_{\varphi(l)}}}{2^{\tau_l}}\|X_l\|\prod_{s=1}^{N-1}\frac{2^{\delta_{\varphi^{s+1}(l)}}}{2^{\tau_{\varphi^s(l)}}}\\
&\le \frac{1}{2^{2(l+1)}}\|X_l\|.
\end{align*}
\end{proof}

\begin{claim}\label{fhc1}
Let $x\in \ell^1$ and $0\le n<l$. For every $j\in \mathopen[0,b_{l+1}-b_l-\delta_l\mathclose]$,
\[\|P_nT^jP_lx\|\le \frac{1}{2^{2(l+1)}}\|P_lx\|.\]
\end{claim}
\begin{proof}
Since $n<l$, we remark that for every $j\in \mathopen[0,b_{l+1}-b_l-\delta_l\mathclose]$, we have
\[\|P_nT^jP_lx\|=\Big\|P_nT^j\Big(\sum_{k=b_l}^{b_{l+1}-1}x_ke_k\Big)\Big\|=\Big\|P_nT^j\Big(\sum_{k=b_l+\delta_l}^{b_{l+1}-1}x_ke_k\Big)\Big\|.\]
Let $\tilde{x}=\sum_{k=b_l+\delta_l}^{b_{l+1}-1}x_ke_k$. We then have
$\| \tilde{X}_l\|\le\|P_lx\|$
and we deduce from Claim \ref{fhc0} that for every $j\in \mathopen[0,b_{l+1}-b_l-\delta_l\mathclose]$,
\[\|P_nT^jP_lx\|=\|P_nT^jP_l\tilde{x}\|\le \frac{1}{2^{2(l+1)}}\|\tilde{X}_l\|\le \frac{1}{2^{2(l+1)}}\|P_lx\|.\]
\end{proof}

\begin{claim}\label{fhc2}
Let $x\in \ell^1$ and $l\ge 0$. Then for every $k\ge 0$, 
\[\frac{\#\{j\le k: \|P_lT^jP_lx\|\ge \frac{\|X_l\|}{2}\}}{k+1}\ge 1 -\frac{2\delta_l}{k+1}-\frac{2\delta_l}{b_{l+1}-b_l}.\]
\end{claim}
\begin{proof}
Let $x\in \ell^1$ and $l\ge 0$.
For any $j\ge 0$, we denote ${i_j:=j\mod (b_{l+1}-b_l)}$ and we denote by $X_{l,j}$ the coordinates of $X_{l}$. A detailed analysis then shows that
\[\|P_lT^jP_lx\|\ge \sum_{m\in \mathopen[b_l,b_{l+1}\mathclose[\backslash I_{i_j}}|X_{l,m}|= \|X_l\|-\sum_{m\in I_{i_j}}|X_{l,m}|,\]
where for any $0\le i< b_{l+1}-b_l$
\[I_i:=\left|
\begin{array}{cl}
\mathopen[b_{l+1}-i,b_{l+1}-i+\delta_l\mathclose[& \quad\text{if}\ i\ge \delta_l \\
 {\mathopen[b_{l+1}-i,b_{l+1}\mathclose[\cup \mathopen[b_l,b_l+\delta_l-i\mathclose[}& \quad\text{if}\ i<\delta_l.
    \end{array}\right.\]
    
Therefore, if we have $\sum_{m\in I_i}|X_{l,m}|<\frac{\|X_l\|}{2}$ for every $0\le i< b_{l+1}-b_l$, we deduce that for every $j\ge 0$
\[\|P_lT^jP_lx\|\ge \|X_l\|-\sum_{m\in I_{i_j}}|X_{l,m}|\ge \|X_l\|-\frac{\|X_l\|}{2}=\frac{\|X_l\|}{2}\]
and thus 
\[\frac{\#\{j\le k: \|P_lT^jP_lx\|\ge \frac{\|X_l\|}{2}\}}{k+1}=1.\]

On the other hand, if there exists $0\le i'< b_{l+1}-b_l$ such that $\sum_{m\in I_{i'}}|X_{l,m}|\ge \frac{\|X_l\|}{2}$ then for every $j\ge 0$ satisfying $I_{i_j}\cap I_{i'}=\emptyset$, we have
\[\|P_lT^jP_lx\|\ge \sum_{m\in \mathopen[b_l,b_{l+1}\mathclose[\backslash I_{i_j}}|X_{l,m}|\ge \sum_{m\in I_{i'}}|X_{l,m}|\ge \frac{\|X_l\|}{2}.\]
In view of the definition of sets $I_i$, we remark that if $(b_{l+1}-b_l)s\ge k$ then \[\#\{j\le k: I_{i_j}\cap I_{i'}\ne \emptyset\}\le 2\delta_l\lceil s\rceil.\] We conclude that for every $k\ge0$
\begin{align*}
\frac{\#\{j\le k: \|P_lT^jP_lx\|\ge \frac{\|X_l\|}{2}\}}{k+1}&\ge \frac{\#\{j\le k: I_{i_j}\cap I_{i'}=\emptyset\}}{k+1}\\
&\ge\frac{(k+1)-2\delta_l\big(1+\frac{k}{b_{l+1}-b_l}\big)}{k+1}\\
&\ge 1-\frac{2\delta_l}{k+1}-\frac{2\delta_l}{b_{l+1}-b_l}.
\end{align*}
\end{proof}

Thanks to Claims~\ref{fhc0}, \ref{fhc1} and \ref{fhc2}, we can show the following result which will directly imply that $T$ is not $\mathcal{U}$-frequently hypercyclic. This claim will also be used in order to prove that $T$ is not distributionally chaotic.

\begin{claim}\label{prelim}
Let $x\in \ell^1\backslash\{0\}$. If for every $n\ge 0$ with $\|X_n\|>0$, there exists $j\ge 0$ such that
\[\sum_{l>n}\|P_nT^jP_lx\|>\frac{\|X_n\|}{4},\]
then there exists $l_0\ge 0$ such that $\|X_{l_0}\|>0$ and \[\underline{\text{\emph{dens}}}\left\{j\ge 0: \|T^jx\|\ge \frac{\|X_{l_0}\|}{4}\right\}=1.\]
\end{claim}
\begin{proof}
Let $x\in \ell^1\backslash\{0\}$. We consider $l_0\ge 0$ such that 
\begin{equation}\label{eq:l0}
\|P_{l_0}x\|\ge \frac{1}{2^{l_0+1}}\|x\|.
\end{equation}
We remark that for every $j\ge 0$, every $n\ge 0$, we have
\begin{equation}\label{useful}
\|T^jx\|\ge \|P_{n}T^jx\|\ge \|P_{n}T^jP_{n}x\|-\sum_{l>n}\|P_{n}T^jP_lx\|.
\end{equation}
Let $j_1:=\min\{j\ge 0:\sum_{l>l_0}\|P_{l_0}T^jP_lx\|> \frac{1}{4}\|X_{l_0}\|\}$, which is well-defined since the set $\{j\ge 0:\sum_{l>l_0}\|P_{l_0}T^jP_lx\|> \frac{1}{4}\|X_{l_0}\|\}$ is non-empty by assumption. We deduce from the definition of $j_1$ that there exists $l_1>l_0$ such that 
\[\|P_{l_0}T^{j_1}P_{l_1}x\|>\frac{2^{l_0}}{2^{l_1+2}}\|X_{l_0}\|\]
 and we deduce that $j_1> b_{l_1+1}-b_{l_1}-\delta_{l_1}$ since, by using Claim~\ref{fhc1} and \eqref{eq:l0}, we have for every $j\in \mathopen[0,b_{l_1+1}-b_{l_1}-\delta_{l_1}\mathclose]$,
\[\|P_{l_0}T^{j}P_{l_1}x\|\le \frac{1}{2^{2(l_1+1)}}\|P_{l_1}x\|\le \frac{1}{2^{2(l_1+1)}}\|x\|
\le \frac{2^{l_0+1}}{2^{2(l_1+1)}}\|P_{l_0}x\|\le \frac{2^{l_0}}{2^{l_1+2}}\|X_{l_0}\|
.\]
On the other hand, by using Claim~\ref{fhc0}, we get
\[\|X_{l_1}\|\ge 2^{2(l_1+1)}\|P_{l_0}T^{j_1}P_{l_1}x\|\ge 2^{l_1+l_0}\|X_{l_0}\|\ge \|X_{l_0}\|.\]

If we now let $j_2:=\min\{j\ge 0:\sum_{l>l_1}\|P_{l_1}T^jP_lx\|> \frac{1}{4}\|X_{l_1}\|\}$, there exists  $l_2>l_1$ such that 
\[\|P_{l_1}T^{j_2}P_{l_2}x\|>\frac{2^{l_1}}{2^{l_2+2}}\|X_{l_1}\|\]
and we deduce as previously that $j_2> b_{l_2+1}-b_{l_2}-\delta_{l_2}$ and $\|X_{l_2}\|\ge \|X_{l_0}\|$.
More generally, by repeating these arguments, we obtain an increasing sequence of integers $(l_n)_{n\ge 0}$ and a sequence of integers $(j_n)_{n\ge 1}$ with $j_n=\min\{j\ge 0:\sum_{l>l_{n-1}}\|P_{l_{n-1}}T^jP_lx\|> \frac{1}{4}\|X_{l_{n-1}}\|\}$ such that for every $n\ge 1$, 
\[j_n> b_{l_n+1}-b_{l_n}-\delta_{l_n}\quad \text{ and }\quad \|X_{l_n}\|\ge \|X_{l_0}\|.\] In particular, we deduce from \eqref{useful} that for every $n\ge 0$, every $j<j_{n+1}$,
\begin{equation}
\label{Tj}
\|T^jx\|\ge \|P_{l_{n}}T^jP_{l_{n}}x\|-\frac{1}{4}\|X_{l_{n}}\|.
\end{equation}



We remark that $\lim j_n=\infty$ by \eqref{b1}. Therefore, if we consider an integer $k\ge j_1$ and if we let $n_k=\min\{n\ge 0: k<j_{n+1}\}$, then we have $k\in \mathopen[j_{n_k},j_{n_k+1}\mathclose[$ and $n_k\to \infty$. We deduce from Claim~\ref{fhc2} that
\[\frac{\#\{j\le k: \|P_{l_{n_k}}T^{j}P_{l_{n_k}}x\|\ge \frac{\|X_{l_{n_k}}\|}{2}\}}{k+1}\ge 1 -\frac{2\delta_{l_{n_k}}}{k+1}-\frac{2\delta_{l_{n_k}}}{b_{l_{n_k}+1}-b_{l_{n_k}}}.\]
We then get thanks to \eqref{Tj}
\[\frac{\#\{j\le k: \|T^jx\|\ge \frac{\|X_{l_{n_k}}\|}{4}\}}{k+1}\ge 1 -\frac{2\delta_{l_{n_k}}}{k+1}-\frac{2\delta_{l_{n_k}}}{b_{l_{n_k}+1}-b_{l_{n_k}}}.\]
Finally, since $k\ge j_{n_k}> b_{l_{n_k}+1}-b_{l_{n_k}}-\delta_{l_{n_k}}$ and since $\|X_{l_{n_k}}\|\ge \|X_{l_0}\|$, we deduce that 
\[\frac{\#\{j\le k: \|T^jx\|\ge \frac{\|X_{l_0}\|}{4}\}}{k+1}\ge 1 -\frac{2\delta_{l_{n_k}}}{b_{l_{n_k}+1}-b_{l_{n_k}}-\delta_{l_{n_k}}}-\frac{2\delta_{l_{n_k}}}{b_{l_{n_k}+1}-b_{l_{n_k}}}\]
and, using \eqref{b2}, we get
\[\underline{\text{dens}}\Big\{j\ge 0: \|T^jx\|\ge \frac{\|X_{l_0}\|}{4}\Big\}=1.\]
\end{proof}
\newpage
\begin{claim}\label{Ufhc}
$T$ is not $\U$-frequently hypercyclic.
\end{claim}
\begin{proof}
Let $x$ be a hypercyclic vector for $T$. We first show that for every $n\ge 0$ there exists $j\ge 0$ such that
\[\sum_{l>n}\|P_nT^jP_lx\|> \frac{\|X_{n}\|}{4}.\] 
Let $n\ge 0$ and $K:=\sup_{j\ge 0}\|P_nT^jP_nx\|$, which is finite since $P_nx$ is periodic.  If we consider $j\ge 0$ such that $\|T^jx-(K+\frac{1}{2}\|X_n\|)e_{b_n}\|<\frac{1}{4}\|X_n\|$, then we have
\[\sum_{l>n}\|P_nT^jP_lx\|\ge \|P_nT^jx\|-\|P_nT^jP_nx\|>\Big(K+\frac{\|X_n\|}{4}\Big)-K=\frac{\|X_n\|}{4}.\]

We therefore deduce from Claim~\ref{prelim} that there exists $l_0\ge 0$ such that $\|X_{l_0}\|>0$ and
\[\underline{\text{dens}}\Big\{j\ge 0: \|T^jx\|\ge \frac{\|X_{l_0}\|}{4}\Big\}=1.\]
This means that
\[\overline{\text{dens}}\Big\{j\ge 0: \|T^jx\|< \frac{\|X_{l_0}\|}{4}\Big\}=0.\]
We conclude that $x$ is not a $\mathcal{U}$-frequently hypercyclic vector for $T$ since $\|X_{l_0}\|>0$ and thus that $T$ is not $\mathcal{U}$-frequently hypercyclic.
\end{proof}

\subsection{$T$ is not distributionally chaotic}\label{Tpasdist}

By using Claims~\ref{fhc2} and \ref{prelim}, we can also show the following result.

\begin{claim}\label{cool}
Let $x\in \ell^1$. If $x\ne 0$, then there exists $\tau>0$ such that
\[\underline{\text{\emph{dens}}}\{j\ge 0: \|T^jx\|\ge \tau \}>0.\]
\end{claim}
\begin{proof}
Let $x\in \ell^1\backslash\{0\}$. If for every $n\ge 0$ with $\|X_n\|>0$, there exists $j\ge 0$ such that
\[\sum_{l>n}\|P_{n}T^jP_{l}x\|> \frac{\|X_{n}\|}{4},\]
then we deduce from Claim~\ref{prelim} that there exists $l_0\ge 0$ such that $\|X_{l_0}\|>0$ and
\[\underline{\text{dens}}\Big\{j\ge 0: \|T^jx\|\ge \frac{\|X_{l_0}\|}{4}\Big\}=1.\]

On the other hand, if there exists $n\ge 0$ with $\|X_n\|>0$ such that for any $j\ge 0$, we have
\[\sum_{l>n}\|P_{n}T^jP_{l}x\|\le \frac{\|X_{n}\|}{4},\]
then we deduce that for any $j\ge 0$
\[\|T^jx\|\ge \|P_{n}T^jP_{n}x\|-\sum_{l>n}\|P_{n}T^jP_lx\|\ge \|P_{n}T^jP_{n}x\|-\frac{\|X_{n}\|}{4}\]
and we conclude by Claim~\ref{fhc2} and \eqref{b1} that 
\begin{align*}
\underline{\text{dens}}\Big\{j\ge 0: \|T^jx\|\ge \frac{\|X_{n}\|}{4}\Big\}&\ge 
\underline{\text{dens}}\Big\{j\ge 0: \|P_{n}T^jP_{n}x\|\ge \frac{\|X_{n}\|}{2}\Big\}\\
&\ge 1 -\frac{2\delta_n}{b_{n+1}-b_n}>0.
\end{align*}
\end{proof}

We can now easily deduce from Claim~\ref{cool} that $T$ is not distributionally chaotic.

\begin{claim}
$T$ is not distributionally chaotic.
\end{claim}
\begin{proof}
Assume that $T$ is distributionally chaotic. By definition of distributional chaos,
there then exists an uncountable subset $\Gamma\subset X$ such that for every $x,y\in \Gamma$, $x\ne y$, for every $\tau>0$, we have
\[\overline{\text{dens}}\{n\ge 0:\|T^nx-T^ny\|<\tau\}=1.\]
In particular, it means that there exists $x\ne y$ such that for every $\tau>0$
\[\underline{\text{dens}}\{n\ge 0:\|T^n(x-y)\|\ge\tau\}=0,\]
which contradicts Claim~\ref{cool}.
\end{proof}

\section{Conclusion and remarks}

In view of the obtained results in this paper, we can summarize the links between the main notions in Linear Dynamics as depicted in Figure~\ref{fig}.

\begin{figure}[H]
    \begin{center}
      \begin{tikzpicture}[xscale=1,yscale=1.5,>=stealth',shorten >=1pt]
       
        \path (2,5) node[] (q00) {(1) Frequently hypercyclic};
        \path (2,4.2) node[] (q0) {(3) $\U$-frequently hypercyclic};
        \path (0,3.4) node[] (q1) {(4) reiteratively hypercyclic};

        \path (-2,4.2) node[] (q2) {(2) Chaotic};
        \path (4,3.4) node[] (q2lm) {(5) Mixing};

        \path (2,2.6) node[] (q4) {(6) Weakly mixing};
        \path (2,1.8) node[] (q5) {(7) Hypercyclic};

        \draw[double,arrows=->] (q00) -- (q0);
        \draw[double,arrows=->] (q0) -- (q1);

        \draw[double,arrows=->] (q1) -- (q4);
        \draw[double,arrows=->] (q2lm) -- (q4);

        \draw[double,arrows=->] (q2) -- (q1);
        \draw[double,arrows=->] (q4) -- (q5);

        \end{tikzpicture}
    \end{center}
    \caption{Links between the different notions in Linear Dynamics}
\label{fig}
\end{figure}
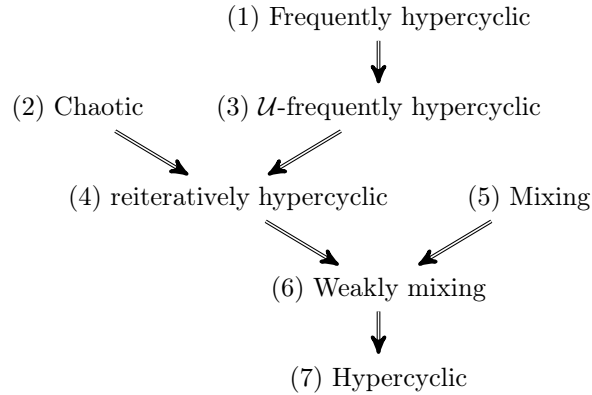

Indeed, the implication $(4)  \Rightarrow (6)$ has been proved in \cite{Bes}, the implication ${(2)\Rightarrow (4)}$ follows from Theorem~\ref{chaosreit} and each of the other implications is obvious by definition. Moreover, there are no other implications since there exist:
\begin{itemize}

\item a weakly mixing weighted shift on $\ell^p$ which is not mixing~\cite{Costakis},
\item a chaotic operator which is not mixing~\cite{Badea},
\item a frequently hypercyclic weighted shift on $c_0$ which is neither chaotic nor mixing~\cite{BGrivaux},
\item a hypercyclic operator which is not weakly mixing~\cite{2Read}.
\item a $\U$-frequently hypercyclic weighted shift on $c_0$ which is not frequently hypercyclic~\cite{BayartR},
\item a reiteratively hypercyclic weighted shift on $c_0$ which is not $\U$-frequently hypercyclic~\cite{Bes}
\item a mixing weighted shift on $\ell^p$ which is not reiteratively hypercyclic \cite{Bes},
\item a chaotic operator which is not $\U$-frequently hypercyclic (Theorem~\ref{chaosfhc}).
\end{itemize}

On the other hand, thanks to Bonet and Peris~\cite{4Bonet2}, we know that every separable infinite-dimensional Fréchet space supports a hypercyclic operator and even a mixing operator~\cite{4Grivaux}. However, the situation is different for chaos and frequent hypercyclicity. Indeed, there exist separable infinite-dimensional Banach spaces which support no chaotic operator~\cite{Bonet} and no frequently hypercyclic operator~\cite{Shkarin}. Theorem~\ref{chaosfhc} then leads to the following question.

\begin{question}
Does there exist a separable infinite-dimensional Banach space which supports a chaotic operator and no frequently hypercyclic operator?
\end{question}

One can also wonder:

\begin{question}
On which spaces does there exist a chaotic operator which is not $\mathcal{U}$-frequently hypercyclic/frequently hypercyclic/distributionally chaotic?
\end{question}

We already know that there exists a chaotic operator $T$ on $\ell^1$ which is neither $\U$-frequently hypercyclic nor distributionally chaotic (Theorem~\ref{chaosfhc}) and we can show that there also exists such an operator on $c_0$ and on $\ell^p$ for every $p\in [1,\infty[$. Indeed, the operator $T$ considered in Theorem~\ref{chaosfhc} is in fact a continuous operator on $c_0$ and on $l^p$ for every $p\in [1,\infty[$ and we deduce from Claims~\ref{claimper},~\ref{claimper2},~\ref{hyp0}~and~\ref{hyp} that $T$ is also a chaotic operator on $c_0$ and on $l^p$ for every $p\in [1,\infty[$. Moreover, since for every $l\ge 0$, we have $\|X_l\|_1\le (b_{l+1}-b_l)\|X_l\|_\infty$ and $\|P_l x\|_1\le (b_{l+1}-b_l)\|P_l x\|_\infty$, Claims~\ref{fhc0}~and~\ref{fhc1} remain true on $c_0$ and on $l^p$ for every $p\in [1,\infty[$ if we replace \eqref{deltatau2} by 
\begin{equation}
\frac{2^{\delta_{n-1}}(b_{n+1}-b_n)}{2^{\tau_n}}\le \frac{1}{2^{2(n+1)}}.
\label{41}
\end{equation}
We can then show that if $T$ satisfies \eqref{41}, $T$ is neither $\U$-frequently hypercyclic nor distributionally chaotic on $c_0$ and on $l^p$. Moreover, we remark that \eqref{41} is satisfied if we consider, as previously, for every $n\ge 1$
\[\tau_n= 4^{n+1},\quad \delta_n=2 \tau_n\quad\text{and}\quad b_{n}-b_{n-1}=4^{2n+1}.\]

\end{document}